\documentclass[12pt,reqno]{amsart}
\usepackage{amsmath, amsthm, amssymb, stmaryrd}

\advance \topmargin by -\headheight
\advance \topmargin by -\headsep
     
\evensidemargin \oddsidemargin
\newtheorem{theorem}{Theorem}[section]
\newtheorem{lemma}[theorem]{Lemma}

\theoremstyle{definition}

\theoremstyle{remark}

\numberwithin{equation}{section}

\def\bfx{{\mathbf x}}
\def\bfy{{\mathbf y}}

\def\calC{{\mathcal C}}

\def\dbN{{\mathbb N}}  
\def\dbR{{\mathbb R}}
\def\dbZ{{\mathbb Z}}

\def\grB{{\mathfrak B}}

\def\grm{{\mathfrak m}}\def\grM{{\mathfrak M}}

\def\grB{{\mathfrak B}}

\def\alp{{\alpha}} \def\bfalp{{\boldsymbol \alpha}}
\def\bet{{\beta}}  
\def\gam{{\gamma}} 

\def\del{{\delta}}

\def\eps{\varepsilon}

\def\le{\leqslant} \def\ge{\geqslant}

\def\d{{\,{\rm d}}}

\begin{document}
\title[A paucity problem]{A paucity problem for certain triples of diagonal equations}
\author[J\"org Br\"udern]{J\"org Br\"udern}
\address{Mathematisches Institut, Bunsenstrasse 3--5, D-37073 G\"ottingen, Germany}
\email{joerg.bruedern@mathematik.uni-goettingen.de}
\author[Trevor D. Wooley]{Trevor D. Wooley}
\address{Department of Mathematics, Purdue University, 150 N. University Street, West 
Lafayette, IN 47907-2067, USA}
\email{twooley@purdue.edu}
\date{}
\dedicatory{}
\begin{abstract}We consider certain systems of three linked simultaneous diagonal 
equations in ten variables with total degree exceeding five. By means of a complification 
argument, we obtain an asymptotic formula for the number of integral solutions of this 
system of bounded height that resolves the associated paucity problem.
\end{abstract}
\maketitle

\section{Introduction} In this note we investigate the simultaneous Diophantine equations
\begin{equation}\label{1.1}
\sum_{i=1}^5 (x_i^k-y_i^k)=\sum_{i=1}^3 (x_i^n-y_i^n)=\sum_{i=4}^5 
(x_i^m-y_i^m)=0,
\end{equation}
focusing our attention on the number $N_{k,m,n}(B)$ of integral solutions $\bfx,\bfy$ of 
this system satisfying $1\le x_i,y_i\le B$ $(1\le i\le 5)$. These equations admit the 
diagonal solutions with
\[
\{x_1,x_2,x_3\}=\{y_1,y_2,y_3\}\quad \text{and}\quad \{x_4,x_5\}=\{y_4,y_5\},
\]
contributing an amount
\begin{equation}\label{1.2}
T(B)=(3!B^3+O(B^2))(2!B^2+O(B))=12B^5+O(B^4)
\end{equation}
to the total count $N_{k,m,n}(B)$. Whether or not one should expect an abundance of 
non-diagonal solutions to the system (\ref{1.1}) depends on the triple $(k,m,n)$. Excluding 
from consideration the degenerate cases in which $k\in \{m,n\}$, the goal of this paper is 
the characterisation of the triples $(k,m,n)$ for which there is a paucity of non-diagonal 
solutions.

\begin{theorem}\label{theorem1.1}
Suppose that $(k,m,n)\ne (3,1,1)$, and further that neither $(k,n)=(2,1)$ nor 
$(k,n)=(1,2)$. Then, for any positive number $\del$ with $\del<1/12$, one has
\[
N_{k,m,n}(B)=12B^5+O(B^{5-\del}).
\]
\end{theorem}

In \S2 we show that when $(k,n)$ is either $(2,1)$ or $(1,2)$, one has
\begin{equation}\label{1.3}
N_{k,m,n}(B)\gg B^5\log (2B).
\end{equation}
Moreover, as a consequence of our earlier work \cite{BW2019}, one may show that
\begin{equation}\label{1.4}
N_{3,1,1}(B)-T(B)\gg B^5.
\end{equation}
For all other triples $(k,m,n)$ with $k\not\in \{m,n\}$, it follows from Theorem 
\ref{theorem1.1} that
\[
N_{k,m,n}(B)=T(B)+o(T(B)),
\]
whence there is a paucity of non-diagonal solutions in the system (\ref{1.1}).\par

It would be possible to extend our methods from the counting problem of estimating 
$N_{k,m,n}(B)$ to the associated problem of estimating the quantity $N_{k,m,n}^\pm(B)$, 
 wherein the solutions of (\ref{1.1}) are counted with $|x_i|,|y_i|\le B$. By weakening the 
condition $1\le x_i,y_i\le B$ so as to include also negative solutions of (\ref{1.1}), one 
encounters additional linear spaces of solutions, and thus the asymptotic formula 
$N_{k,m,n}(B)=12B^5+O(B^{5-\del})$ must be replaced by the relation
\[
N_{k,m,n}^\pm (B)=\rho_{k,m,n}B^5+O(B^{5-\del}),
\]
where $\rho_{k,m,n}$ is a certain positive integer depending on the respective parities of 
$k$, $m$ and $n$. The exposition of our ideas would be significantly complicated and 
lengthened by the associated combinatorial details, as much by additional notation as 
anything of substance. Dedicated readers may check the details for themselves.\par

Existing paucity results for a single equation in four variables, and for pairs of equations in 
six variables, play a role in our proof of Theorem \ref{theorem1.1}. However, the ideas 
underlying such results would be insufficient by themselves to deliver the conclusion of 
our theorem. We instead reach for the strategy described in our recent work 
\cite{BW2019} concerning diagonal cubic equations with two linear slices. This work, which 
addresses the case $(k,m,n)=(3,1,1)$ of the system (\ref{1.1}), and yields an asymptotic 
formula confirming the lower bound (\ref{1.4}), involves an application of the 
Hardy-Littlewood method in combination with a certain complification argument. Our 
approach in the present note once again highlights the opportunity for powerful interplay 
between equations to be exploited when analysing systems of many diagonal equations. 
We refer the reader to \cite{BW2003} and \cite{BW2016a} for earlier instances in which 
such an observation has been utilised.\par

This paper is organised as follows. In \S2 we introduce the infrastructure required for the 
subsequent discussion, justifying {\it en passant} the relations (\ref{1.3}) and (\ref{1.4}). 
A paucity result involving four $m$-th powers in \S3 handily disposes of triples $(k,m,n)$ 
with $m\ge 3$. We examine in \S4 an upper bound for the number of non-zero integers 
$h$ represented by the trailing block
\begin{align*}
x_4^k-y_4^k+x_5^k-y_5^k&=h\\
x_4^m-y_4^m+x_5^m-y_5^m&=0
\end{align*}
in (\ref{1.1}). Thus equipped, we dispose of triples $(k,m,n)$ with $n\ge 3$. The 
complification process comes into play in \S\S5-7. Here, an application of Cauchy's inequality 
relates non-diagonal solutions in the system (\ref{1.1}) to the number of solutions of a 
related system in $12$ variables having respective degrees $k$, $n$ and $n$. The simplest 
application of this idea handles triples $(k,m,n)$ in \S5 with $n=2$. Then, in \S6, a similar 
argument takes care of triples $(k,m,n)$ with $n=1$ and $k\ge 4$. Our final case awaits 
our attention in \S7, namely that with $(k,m,n)=(3,2,1)$. In this situation we are forced to 
apply a crude version of the Hardy-Littlewood method in concert with complification, 
drawing inspiration from aspects of our treatment of the case $(k,m,n)=(3,1,1)$ in 
\cite{BW2019}.\par  

Our basic parameter is $B$, a sufficiently large positive number. Whenever $\eps$ appears 
in a statement, either implicitly or explicitly, we assert that the statement holds for each 
$\eps>0$. In this paper, implicit constants in Vinogradov's notation $\ll$ and $\gg$ may 
depend on $\eps$, $k$, $m$ and $n$. We make frequent use of vector notation in the form 
$\bfx=(x_1,\ldots,x_r)$. Here, the dimension $r$ depends on the course of the argument. 
Finally, we write $e(z)$ for $e^{2\pi iz}$.\par
 
\noindent {\bf Acknowledgements:} The authors acknowledge support by Akademie der 
Wissenschaften zu G\"ottingen and Deutsche Forschungsgemeinschaft Project Number 
255083470. The second author's work is supported by the NSF Focused Research Group 
grant DMS-1854398 and DMS-2001549. 

\section{Infrastructure and the excluded cases} We fix a triple $(k,m,n)$ with 
$k\not\in \{m,n\}$. Defining the exponential sum
\[
f_{k_1,k_2}(\alp_1,\alp_2)=\sum_{1\le x\le B}e(\alp_1x^{k_1}+\alp_2x^{k_2}),
\]
it follows via orthogonality that
\begin{equation}\label{2.1}
N_{k,m,n}(B)=\int_{[0,1)^3}|f_{k,n}(\alp,\bet)^6f_{k,m}(\alp,\gam)^4|\d\bfalp ,
\end{equation}
where we use $\bfalp$ to denote $(\alp,\bet,\gam)$.\par

For the time-being, it suffices to decompose the mean value (\ref{2.1}) by introducing the 
auxiliary integrals $u(h)=u_{k,n}(h)$ and $v(h)=v_{k,m}(h)$, defined by
\begin{equation}\label{2.2}
u(h)=\int_{[0,1)^2}|f_{k,n}(\alp,\bet)|^6e(-h\alp)\d\alp \d\bet 
\end{equation}
and
\begin{equation}\label{2.3}
v(h)=\int_{[0,1)^2}|f_{k,m}(\alp,\gam)|^4e(-h\alp)\d\alp \d\gam .
\end{equation}
Here, by orthogonality, one sees that $u(h)$ counts the representations of the integer $h$ 
in the form
\begin{equation}\label{2.4}
\sum_{i=1}^3(x_i^k-y_i^k)=h
\end{equation}
subject to
\begin{equation}\label{2.5}
\sum_{i=1}^3(x_i^n-y_i^n)=0,
\end{equation}
with $1\le x_i,y_i\le B$ $(1\le i\le 3)$. Likewise, we find from (\ref{2.3}) that $v(h)$ counts 
the number of solutions of the system
\begin{align}
x_1^k+x_2^k-y_1^k-y_2^k&=h,\label{2.6}\\
x_1^m+x_2^m-y_1^m-y_2^m&=0,\label{2.7}
\end{align}
with $1\le x_i,y_i\le B$ $(i=1,2)$. Thus, we see that
\begin{equation}\label{2.8}
N_{k,m,n}(B)=\sum_{|h|\le 2B^k}u(h)v(h).
\end{equation}

\par We pause at this point to remark that, as a consequence of the work of the first 
author joint with Blomer \cite{BB2010}, one has the asymptotic formula
\[
u_{2,1}(0)=u_{1,2}(0)=\frac{18}{\pi^2}B^3\log B+O(B^3).
\]
Moreover, when $m\ne k$, it follows that whenever $\bfx,\bfy\in \dbN^2$ and
\begin{align*}
x_1^k+x_2^k&=y_1^k+y_2^k\\
x_1^m+x_2^m&=y_1^m+y_2^m,
\end{align*}
then $\{x_1,x_2\}=\{y_1,y_2\}$. This assertion may be confirmed either by elementary 
arguments, or by reference to \cite{Ste1971}. It follows that one has the asymptotic 
relation
\begin{equation}\label{2.9}
v_{k,m}(0)=2B^2+O(B).
\end{equation}
By substituting these estimates into (\ref{2.8}), we conclude that
\[
N_{2,m,1}(B)\ge u_{2,1}(0)v_{2,m}(0)\gg B^5\log (2B),
\]
and likewise
\[
N_{1,m,2}(B)\ge u_{1,2}(0)v_{1,m}(0)\gg B^5\log (2B).
\]
The lower bound (\ref{1.3}) follows.\par

The relation (\ref{1.4}), though essentially immediate from \cite{BW2019}, merits some 
discussion. In the latter source, it is shown that
\[
N_{3,1,1}^\pm (B)=(45+\calC)(2B)^5+O(B^{5-1/200}),
\]
where $\calC>0$ is a product of local densities. Here, the constant $45$ is associated with 
the number of linear spaces of solutions of the system (\ref{1.1}) in the case 
$(k,m,n)=(3,1,1)$ generalising the diagonal solutions relevant to our examination of 
$N_{3,1,1}(B)$. Excluding solutions of (\ref{1.1}) involving negative integers simplifies the 
analysis of \cite{BW2019} somewhat, and thus one may proceed at a pedestrian pace to 
obtain the asymptotic formula
\[
N_{3,1,1}(B)=(12+\calC')B^5+O(B^{5-1/200}),
\]
where $\calC'>0$ is the product of local densities associated with the system (\ref{1.1}) in 
the positive sector. In particular, in view of (\ref{1.2}), one has the relation
\[
N_{3,1,1}(B)-T(B)\sim \calC'B^5,
\]
confirming the lower bound (\ref{1.4}).\par

Having discussed the excluded cases, we proceed in the remainder of the paper under the 
assumption that
\begin{equation}\label{2.10}
(k,n)\not\in \{ (2,1), (1,2)\}\quad \text{and}\quad (k,m,n)\ne (3,1,1).
\end{equation}
Since also $k\not\in \{m,n\}$, we may assume that one of the following holds:
\begin{enumerate}
\item[(i)] $m\ge 3$;
\item[(ii)] $m\in \{1,2\}$ and $n\ge 3$;
\item[(iii)] $m\in \{1,2\}$, $n=2$ and $k\ge 3$;
\item[(iv)] $m\in\{1,2\}$, $n=1$ and $k\ge 4$;
\item[(v)] $(k,m,n)=(3,2,1)$.
\end{enumerate}

\par Notice that the first condition in (\ref{2.10}) ensures, via available paucity results, that
\begin{equation}\label{2.11}
u_{k,n}(0)=6B^3+O(B^{8/3}).
\end{equation}
A convenient reference for a result of this strength may be obtained by combining 
\cite[Theorem 1.2]{VW1995}, when $(k,n)=(3,1)$, with \cite[Theorem 1]{Woo1996}, when 
$(k,n)=(3,2)$, and \cite[Corollary 0.3]{Sal2007}, when $k\ge 4$. By combining this 
conclusion with (\ref{2.9}), we see from (\ref{2.8}) that
\begin{align}
N_{k,m,n}(B)&=u_{k,n}(0)v_{k,m}(0)+\sum_{1\le |h|\le 2B^k}u_{k,n}(h)v_{k,m}(h)\notag 
\\
&=12B^5+\sum_{1\le |h|\le 2B^k}u_{k,n}(h)v_{k,m}(h)+O(B^{14/3}).\label{2.12}
\end{align}

\par Our task in the remaining sections is to analyse the sum on the right hand side of 
(\ref{2.12}). We claim that for the triples $(k,m,n)$ classified in the cases (i) to (v) above, 
for any positive number $\eta<1/12$, one has
\begin{equation}\label{2.13}
\sum_{1\le |h|\le 2B^k}u_{k,n}(h)v_{k,m}(h)\ll B^{5-\eta}.
\end{equation}
By substituting this estimate into (\ref{2.12}), we infer that
\[
N_{k,m,n}(B)=12B^5+O(B^{5-\eta}),
\]
and the conclusion of Theorem \ref{theorem1.1} follows.

\section{Paucity for four $m$-th powers} Our first step towards the proof of Theorem 
\ref{theorem1.1} is the discussion of triples $(k,m,n)$ of type (i), with $m\ge 3$. Here we 
make use of available upper bounds for the number $w_m(B)$ of solutions $\bfx,\bfy$ of 
the equation
\[
x_1^m+x_2^m=y_1^m+y_2^m,
\]
with $\{x_1,x_2\}\ne \{y_1,y_2\}$ and $1\le x_i,y_i\le B$ $(i=1,2)$.

\begin{lemma}\label{lemma3.1} When $m\ge 3$, one has $w_m(B)\ll B^{5/3+\eps}$.
\end{lemma}

\begin{proof} Perhaps the most convenient references for this conclusion are the papers 
\cite{Hoo1981} and \cite{Hoo1996}, respectively dealing with odd and even exponents 
$m$. More recent developments can be perused in \cite[Corollary 0.2]{Sal2008} and the 
associated discussion.
\end{proof}

We are now equipped to establish the main conclusion of this section.

\begin{lemma}\label{lemma3.2} Suppose that $(k,n)\not\in \{ (2,1), (1,2)\}$ and 
$k\not \in \{m,n\}$. Then whenever $m\ge 3$, one has
\[
N_{k,m,n}(B)-12B^5\ll B^{14/3+\eps}.
\]
\end{lemma}

\begin{proof} Suppose that $\bfx,\bfy$ is a solution of the equations (\ref{2.6}) and 
(\ref{2.7}) with $1\le x_i,y_i\le B$ $(i=1,2)$. When $\{x_1,x_2\}=\{y_1,y_2\}$, one must 
have $h=0$. Thus, when $h\ne 0$, it follows that $\{x_1,x_2\}\ne \{y_1,y_2\}$, whence 
$\bfx,\bfy$ is counted by $w_m(B)$. In particular, one has
\[
\sum_{1\le |h|\le 2B^k}v_{k,m}(h)\le w_m(B)\ll B^{5/3+\eps},
\]
and consequently,
\begin{align}
\sum_{1\le |h|\le 2B^k}u_{k,n}(h)v_{k,m}(h)&\le \left( \sup_h u_{k,n}(h)\right) 
\sum_{1\le |h|\le 2B^k}v_{k,m}(h)\notag \\
&\ll B^{5/3+\eps}\sup_h u_{k,n}(h).\label{3.1}
\end{align}
By the triangle inequality, it follows from (\ref{2.2}) that
\[
\sup_h u_{k,n}(h)\le \int_{[0,1)^2}|f_{k,n}(\alp,\bet)|^6\d\alp \d\bet =u_{k,n}(0).
\]
Thus, on substituting this estimate into (\ref{3.1}) and recalling (\ref{2.11}), we find that
\[
\sum_{1\le |h|\le 2B^k}u_{k,n}(h)v_{k,m}(h)\ll B^{5/3+\eps}\cdot B^3=B^{14/3+\eps}.
\]
The conclusion of the lemma is now immediate from (\ref{2.12}).
\end{proof}

\section{An upper bound for $v(h)$} We next consider triples $(k,m,n)$ of type (ii), with 
$m\in \{1,2\}$ and $n\ge 3$. Our strategy applies bounds for $v_{k,m}(h)$ going beyond 
square-root cancellation.

\begin{lemma}\label{lemma4.1} Suppose that $h\ne 0$. Then
\begin{enumerate}
\item[(i)] when $k>2$, one has $v_{k,1}(h)\ll |h|^\eps$;
\item[(ii)] when $k\ne 2$, one has $v_{k,2}(h)\ll |h|^\eps B^{1+\eps}$;
\item[(iii)] one has $v_{2,1}(h)\ll |h|^\eps B$.
\end{enumerate}
\end{lemma}

\begin{proof} When $m=1$ and $k>2$, the validity of equations (\ref{2.6}) and 
(\ref{2.7}) implies first that
\begin{equation}\label{4.1}
x_2=y_1+y_2-x_1,
\end{equation}
and hence that
\begin{equation}\label{4.2}
(y_1+y_2-x_1)^k-(y_1^k+y_2^k-x_1^k)=h.
\end{equation}
The polynomial on the left hand side here has factors $y_1-x_1$ and $y_2-x_1$, and hence 
there is a polynomial $\Psi_1\in \dbZ[s_1,s_2,s_3]$ of degree $k-2$ for which
\[
(y_1-x_1)(y_2-x_1)\Psi_1(y_1,y_2,x_1)=h.
\]
We therefore see that $y_1-x_1$, $y_2-x_1$ and $\Psi_1(y_1,y_2,x_1)$ are all divisors of 
the non-zero integer $h$. There are $O(|h|^\eps)$ such divisors, say $d_1=y_1-x_1$, 
$d_2=y_2-x_1$ and $d_3=\Psi_1(y_1,y_2,x_1)$, whence
\begin{equation}\label{4.3}
y_1=x_1+d_1,\quad y_2=x_1+d_2\quad \text{and}\quad 
\Psi_1(x_1+d_1,x_1+d_2,x_1)=d_3.
\end{equation}
An examination of (\ref{4.2}) reveals that
\[
(x_1+d_1+d_2)^k-(x_1+d_1)^k-(x_1+d_2)^k+x_1^k=d_1d_2\Psi_1(x_1+d_1,x_1+d_2,x_1),
\]
and so a consideration of the second forward difference polynomial associated with $x^k$ 
reveals that $\Psi_1(x_1+d_1,x_1+d_2,x_1)$ is non-constant as a polynomial in $x_1$. For 
each fixed one of the $O(|h|^\eps)$ possible choices for $d_1$, $d_2$, $d_3$, it therefore 
follows from the final equation in (\ref{4.3}) that there are $O(1)$ possible choices for 
$x_1$. From here, by back substituting first into (\ref{4.3}), and thence into (\ref{4.1}), we 
find that $x_1$, $x_2$, $y_1$, $y_2$ are all fixed. Thus indeed $v_{k,1}(h)\ll |h|^\eps$, 
and the proof of the lemma is complete in case (i).\par

In case (iii) we may proceed in like manner, though in this case we find that $\Psi_1=2$. 
We therefore have as many as $O(B)$ choices remaining available for $x_1$, and so we 
arrive at the weaker upper bound $v_{2,1}(h)\ll |h|^\eps B$.\par

Finally, we examine the situation with $m=2$. Notice first that when $x_1=x_2$ and 
$2x_1^k=h$, then equation (\ref{2.6}) simplifies to $y_1^k+y_2^k=0$, and this is 
impossible because $y_1,y_2\in \dbN$. It follows that either $2x_1^k\ne h$ or 
$2x_2^k\ne h$, and we may assume the latter by symmetry. We now substitute the 
equation
\begin{equation}\label{4.4}
x_2^2=y_1^2+y_2^2-x_1^2
\end{equation}
for (\ref{4.1}), and thus infer that in place of (\ref{4.2}) we have the equation
\begin{align*}
(y_1^2+y_2^2-x_1^2)^k-(y_1^k+y_2^k-x_1^k)^2&=(x_2^2)^k-(x_2^k-h)^2\\
&=h(2x_2^k-h).
\end{align*}
The polynomial on the left hand side here has factors $y_1-x_1$ and $y_2-x_1$, and hence 
there is a polynomial $\Psi_2\in \dbZ[s_1,s_2,s_3]$ of degree $2k-2$ for which
\begin{equation}\label{4.5}
(y_1-x_1)(y_2-x_1)\Psi_2(y_1,y_2,x_1)=h(2x_2^k-h).
\end{equation}

\par For each fixed choice of $x_2$ with $1\le x_2\le B$ in question, we may suppose that 
the right hand side of (\ref{4.5}) is a fixed non-zero integer $N$ with $N\ll |h|B^k$. The 
integers $y_1-x_1$, $y_2-x_1$ and $\Psi_2(y_1,y_2,x_1)$ are each divisors of $N$, and 
hence there are $O(|N|^\eps)$ such divisors, say
\begin{equation}\label{4.6}
d_1=y_1-x_1,\quad d_2=y_2-x_1\quad \text{and}\quad d_3=\Psi_2(y_1,y_2,x_1).
\end{equation}
We now find from (\ref{4.4}) that
\[
(x_1+d_1)^2+(x_1+d_2)^2-x_1^2=x_2^2,
\]
whence
\[
(x_1+d_1+d_2)^2=x_2^2+2d_1d_2.
\]
With $x_2$ already fixed, it follows that for each fixed one of the $O(|h|^\eps B^\eps)$ 
possible choices for $d_1$, $d_2$ and $d_3$, the choice for $x_1$ is fixed by this last 
equation. The variables $y_1$ and $y_2$ are then fixed via (\ref{4.6}), and we conclude 
that $v_{k,2}(h)\ll |h|^\eps B^{1+\eps}$. This completes the proof of part (ii), and hence 
also the lemma.
\end{proof}

The conclusion of Theorem \ref{theorem1.1} in case (ii) is now obtained in a 
straightforward manner by appealing to Hua's lemma.

\begin{lemma}\label{lemma4.2}
Suppose that $m\in \{1,2\}$, $k\not\in \{m,n\}$ and $n\ge 3$. Then one has
\[
N_{k,m,n}(B)-12B^5\ll B^{14/3}.
\]
\end{lemma}

\begin{proof} It follows from Lemma \ref{lemma4.1} that
\[
\max_{1\le |h|\le 2B^k}v_{k,m}(h)\ll B^{1+\eps}.
\]
On substituting this estimate into (\ref{2.12}), we infer that
\begin{equation}\label{4.7}
N_{k,m,n}(B)-12B^5\ll B^{14/3}+B^{1+\eps}\sum_{h\in \dbZ}u_{k,n}(h).
\end{equation}
The last sum here counts the number of integral solutions of the equation
\[
\sum_{i=1}^3(x_i^n-y_i^n)=0,
\]
with $1\le x_i,y_i\le B$ $(1\le i\le 3)$. By orthogonality, an application of Schwarz's 
inequality, and the invocation of Hua's lemma (see \cite[Lemma 2.5]{Vau1997}), we obtain 
the standard estimate
\[
\int_0^1 \Bigl| \sum_{1\le x\le B}e(\alp x^n)\Bigr|^6\d\alp \ll B^{7/2+\eps}
\]
for this quantity. On substituting this upper bound into (\ref{4.7}), we conclude that
\[
N_{k,m,n}(B)-12B^5\ll B^{14/3}+B^{1+\eps}\cdot B^{7/2+\eps}\ll B^{14/3},
\]
and the proof of the lemma is complete.
\end{proof}

\section{A cheap complification argument when $n=2$} Our purpose in this section is to 
handle triples of type (iii), wherein we may suppose that $m\in \{1,2\}$, $n=2$ and 
$k\ge 3$. This we achieve through a complification argument the prosecution of which 
requires several auxiliary mean value estimates. We now supply these estimates.

\begin{lemma}\label{lemma5.1} Suppose that $m\in \{1,2\}$ and $k\ge 3$. Then one has 
\[
\sum_{1\le |h|\le 2B^k} v_{k,m}(h)^2\ll B^{3+\eps}.
\]
\end{lemma}

\begin{proof} By Lemma \ref{lemma4.1}, one has
\begin{equation}\label{5.1}
\sum_{1\le |h|\le 2B^k}v_{k,m}(h)^2\ll B^{m-1+\eps}\sum_{h\in \dbZ}v_{k,m}(h).
\end{equation}
On recalling (\ref{2.6}) and (\ref{2.7}), we see that the sum on the right hand side here 
is bounded above by the number of solutions of the equation
\[
x_1^m+x_2^m=y_1^m+y_2^m,
\]
with $1\le x_i,y_i\le B$. When $m=1$ this is plainly $O(B^3)$, whilst for $m=2$ it follows 
from Hua's lemma that the number of solutions is $O(B^{2+\eps})$ (see 
\cite[Lemma 2.5]{Vau1997}). Thus, in either case, the number of solutions is 
$O(B^{4-m+\eps})$, and we conclude from (\ref{5.1}) that
\[
\sum_{1\le |h|\le 2B^k}v_{k,m}(h)^2\ll B^{m-1+\eps}\cdot B^{4-m+\eps}\ll 
B^{3+2\eps}.
\]
This completes the proof of the lemma.
\end{proof} 

Next we record an upper bound available from recent work associated with Vinogradov's 
mean value theorem.

\begin{lemma}\label{lemma5.2} Suppose that $k\ge 3$. Then one has
\[
\int_0^1\int_0^1 |f_{k,2}(\alp,\bet)|^{12}\d\alp \d\bet \ll B^{7+\eps}.
\]
\end{lemma}

\begin{proof} This is a special case of \cite[Theorem 14.1]{Woo2019}, though the proof 
is simple and transparent enough to provide here in full. Write
\[
c(\bfalp)=\sum_{1\le x\le B}e(\alp x^k+\bet x^2+\gam x).
\]
Then we deduce via the triangle inequality and orthogonality that
\begin{align*}
\int_0^1\int_0^1|f_{k,2}(\alp,\bet)|^{12}\d\alp \d\bet &=\sum_{|l|\le 6B}\int_{[0,1)^3}
|c(\bfalp)|^{12}e(-l\gam )\d\bfalp \\
&\ll B\int_{[0,1)^3}|c(\bfalp)|^{12}\d\bfalp .
\end{align*}
By \cite[Corollary 1.2]{Woo2019}, the last integral is $O(B^{6+\eps})$, and so the desired 
conclusion follows at once.
\end{proof}

Now we come to the proof of Theorem \ref{theorem1.1} in the case (iii).

\begin{lemma}\label{lemma5.3}
Suppose that $m\in \{1,2\}$ and $k\ge 3$. Then one has
\[
N_{k,m,2}(B)-12B^5\ll B^{14/3+\eps}.
\]
\end{lemma}

\begin{proof} An application of Cauchy's inequality in combination with Lemma 
\ref{lemma5.1} yields the bound
\begin{align}
\sum_{1\le |h|\le 2B^k}u_{k,2}(h)v_{k,m}(h)&\le \Bigl( \sum_{1\le |h|\le 2B^k}
v_{k,m}(h)^2\Bigr)^{1/2} \Bigl( \sum_{h\in \dbZ}u_{k,2}(h)^2\Bigr)^{1/2}\notag \\
&\ll (B^{3+\eps})^{1/2}\Bigl( \sum_{h\in \dbZ}u_{k,2}(h)^2\Bigr)^{1/2} .\label{5.2}
\end{align}
On recalling (\ref{2.4}) and (\ref{2.5}), the sum on the right hand side here may be 
reinterpreted in terms of a Diophantine equation. Thus, it follows via orthogonality, 
Schwarz's inequality and symmetry that
\begin{align}
\sum_{h\in \dbZ}u_{k,2}(h)^2&=\int_{[0,1)^3}|f_{k,2}(\alp,\bet)f_{k,2}(\alp,\gam)|^6
\d\bfalp \notag \\
&\le \int_{[0,1)^3}|f_{k,2}(\alp,\bet)^8f_{k,2}(\alp,\gam)^4|\d\bfalp .\label{5.3}
\end{align}

\par Observe that by orthogonality in league with the triangle inequality,
\[
\sup_{\alp \in \dbR}\int_0^1|f_{k,2}(\alp,\gam)|^4\d\gam \le \int_0^1|f_{k,2}(0,\gam)|^4
\d\gam \ll B^{2+\eps},
\]
wherein we interpreted the second integral as the number of solutions of the equation 
$x_1^2+x_2^2=y_1^2+y_2^2$ with $1\le x_i,y_i\le B$, and applied Hua's lemma. 
Returning to (\ref{5.3}) and applying H\"older's inequality, therefore, we find that
\begin{align*}
\sum_{h\in \dbZ}u_{k,2}(h)^2&\ll B^{2+\eps}\int_0^1\int_0^1|f_{k,2}(\alp,\bet)|^8
\d\alp \d\bet \\
&\ll B^{2+\eps}I_6^{2/3}I_{12}^{1/3},
\end{align*}
where
\[
I_t=\int_0^1\int_0^1|f_{k,2}(\alp,\bet)|^t\d\alp \d\bet .
\]
By orthogonality, we see from (\ref{2.11}) that
\[
I_6=u_{k,2}(0)\ll B^3,
\]
whilst Lemma \ref{lemma5.2} delivers the bound $I_{12}\ll B^{7+\eps}$. Thus we deduce 
that
\[
\sum_{h\in \dbZ}u_{k,2}(h)^2\ll B^{2+\eps}(B^3)^{2/3}(B^{7+\eps})^{1/3}
\ll B^{19/3+2\eps}.
\]

\par Finally, by substituting the last bound into (\ref{5.2}), we arrive at the estimate
\[
\sum_{1\le |h|\le 2B^k}u_{k,2}(h)v_{k,m}(h)\ll (B^{3+\eps})^{1/2}
(B^{19/3+\eps})^{1/2}\ll B^{14/3+\eps}.
\]
This, when substituted into (\ref{2.12}), delivers the relation
\[
N_{k,m,2}(B)-12B^5\ll B^{14/3+\eps},
\]
and this completes the proof of the lemma.
\end{proof}

\section{A cheap complification argument when $n=1$ and $k\ge 4$} The analysis of 
triples of type (iv) is similar to that applied in the previous section for triples of type (iii). We 
now suppose that $n=1$ and $k\ge 4$, however, which prevents appeal to the relatively 
powerful mean value estimates for quadratic Weyl sums available when $n=2$. We again 
begin with an auxiliary mean value estimate.

\begin{lemma}\label{lemma6.1} Suppose that $k\ge 4$. Then one has
\[
\int_{[0,1)^3}|f_{k,1}(\alp,\bet)^6f_{k,1}(\alp,\gam)^{14}|\d\bfalp \ll B^{14+\eps}.
\]
\end{lemma}

\begin{proof} Write
\[
F(\bfalp)=|f_{k,1}(\alp,\bet)^6f_{k,1}(\alp,\gam)^{14}|.
\]
Then by applying the elementary inequality
\[
|z_1\ldots z_r|\le |z_1|^r+\ldots +|z_r|^r,
\]
we see that
\[
F(\bfalp)\ll |f_{k,1}(\alp,\bet)^{18}f_{k,1}(\alp,\gam)^2|+
|f_{k,1}(\alp,\bet)^2f_{k,1}(\alp,\gam)^{18}| .
\]
Thus, by symmetry and orthogonality, we find that
\begin{align*}
\int_{[0,1)^3}F(\bfalp)\d\bfalp &\ll \int_0^1\int_0^1|f_{k,1}(\alp,\bet)|^{18}
\int_0^1|f_{k,1}(\alp,\gam)|^2\d\gam \d\bet \d\alp \\
&\le B\int_0^1\int_0^1 |f_{k,1}(\alp,\bet)|^{18}\d\alp \d\bet .
\end{align*}
The last integral is the subject of \cite[Lemma 5]{BR2015}, which shows that
\[
\int_0^1\int_0^1|f_{k,1}(\alp,\bet)|^{2^j+2}\d\alp \d\bet \ll B^{2^j-j+1+\eps}\quad 
(2\le j\le k).
\]
Thus, by applying this estimate with $j=4$, we deduce that
\[
\int_{[0,1)^3}|f_{k,1}(\alp,\bet)^6f_{k,1}(\alp,\gam)^{14}|\d\bfalp \ll B\cdot B^{13+\eps}
=B^{14+\eps}.
\]
This completes the proof of the lemma.
\end{proof}

We may now tackle the main conclusion of this section.

\begin{lemma}\label{lemma6.2}
Suppose that $m\in \{1,2\}$ and $k\ge 4$. Then one has
\[
N_{k,m,1}(B)-12B^5\ll B^{49/10+\eps}.
\]
\end{lemma}

\begin{proof} Just as in the initial stages of the proof of Lemma \ref{lemma5.3}, an 
application of Cauchy's inequality in combination with Lemma \ref{lemma5.1} yields the 
bound
\begin{equation}\label{6.1}
\sum_{1\le |h|\le 2B^k}u_{k,1}(h)v_{k,m}(h)\ll (B^{3+\eps})^{1/2}\Bigl( 
\sum_{h\in \dbZ}u_{k,1}(h)^2\Bigr)^{1/2}.
\end{equation}
The sum on the right hand side here may be again reinterpreted as the number of solutions 
of a Diophantine system, and thence by orthogonality and H\"older's inequality we obtain
\begin{equation}\label{6.2}
\sum_{h\in \dbZ}u_{k,1}(h)^2=\int_{[0,1)^3}|f_{k,1}(\alp,\bet)f_{k,1}(\alp,\gam)|^6
\d\bfalp \le T_1^{4/5}T_2^{1/5},
\end{equation}
where
\[
T_1=\int_{[0,1)^3}|f_{k,1}(\alp,\bet)^6f_{k,1}(\alp,\gam)^4|\d\bfalp
\]
and
\[
T_2=\int_{[0,1)^3}|f_{k,1}(\alp,\bet)^6f_{k,1}(\alp,\gam)^{14}|\d\bfalp .
\]

\par By orthogonality, one sees that $T_1=N_{k,1,1}(B)$, whilst by Lemma \ref{lemma6.1} 
we have $T_2\ll B^{14+\eps}$. On substituting these estimates into (\ref{6.2}) and thence 
into (\ref{6.1}), we see that
\[
\sum_{1\le |h|\le 2B^k}u_{k,1}(h)v_{k,m}(h)\ll (B^{3+\eps})^{1/2}\left( 
N_{k,1,1}(B)\right)^{2/5}(B^{14+\eps})^{1/10},
\]
so that, as a consequence of (\ref{2.12}),
\begin{equation}\label{6.4}
N_{k,m,1}(B)-12B^5\ll B^{14/3}+B^{29/10+\eps}\left( N_{k,1,1}(B)\right)^{2/5}.
\end{equation}
This estimate applies when $m=1$, and hence in particular one finds that
\[
N_{k,1,1}(B)\ll B^5+B^{29/10+\eps}\left( N_{k,1,1}(B)\right)^{2/5},
\]
whence $N_{k,1,1}(B)\ll B^5$. By substituting this upper bound back into (\ref{6.4}), we 
infer that
\[
N_{k,m,1}(B)-12B^5\ll B^{14/3}+B^{29/10+\eps}(B^5)^{2/5}\ll B^{49/10+\eps}.
\]
This completes the proof of the lemma.
\end{proof}

\section{An application of the Hardy-Littlewood method} The final case (v) concerns the 
only remaining triple not already covered in cases (i) to (iv), namely the triple 
$(k,m,n)=(3,2,1)$. For this we must modify the treatment of \S6 by introducing some crude 
estimates pertaining to the minor arcs of a Hardy-Littlewood dissection.\par

We define our Hardy-Littlewood dissection as follows. Take $\del$ to be any positive number 
with $\del<1/3$, and let $\grM$ denote the union of the intervals
\begin{equation}\label{7.1}
\grM(q,a)=\{ \alp \in [0,1):|q\alp -a|\le B^{\del-3}\},
\end{equation}
with $0\le a\le q\le B^\del$ and $(a,q)=1$. The complementary set of minor arcs is then 
$\grm=[0,1)\setminus \grM$. On writing
\begin{equation}\label{7.2}
N(B;\grB)=\int_\grB \int_0^1\int_0^1 |f_{3,1}(\alp,\bet)^6f_{3,2}(\alp,\gam)^4|\d\gam 
\d\bet \d\alp ,
\end{equation}
we see that
\begin{equation}\label{7.3}
N_{3,2,1}(B)=N(B;\grM)+N(B;\grm).  
\end{equation}
We also define the auxiliary integral
\[
u(h;\grB)=\int_\grB \int_0^1 |f_{3,1}(\alp,\bet)|^6e(-h\alp)\d\bet \d\alp .
\]
In view of the definition of $v(h)=v_{3,2}(h)$ via (\ref{2.6}) and (\ref{2.7}), we then have
\[
\sum_{|h|\le 2B^3}u(h;\grB)v(h)=\sum_{\bfx,\bfy}\int_\grB \int_0^1 
|f_{3,1}(\alp,\bet)|^6e(-(x_1^3+x_2^3-y_1^3-y_2^3)\alp)\d\bet \d\alp ,
\]
where the summation over $\bfx$ and $\bfy$ is subject to the conditions $1\le x_i,y_i\le B$ 
$(i=1,2)$ and $x_1^2+x_2^2=y_1^2+y_2^2$. Thus, by employing orthogonality and 
recalling (\ref{7.2}), we discern that
\begin{equation}\label{7.4}
N(B;\grB)=\sum_{|h|\le 2B^3}u(h;\grB)v(h).
\end{equation}

\par In general terms, our strategy makes use of a complification step resembling that used 
in both \S\S5 and 6. However, our use of a Hardy-Littlewood dissection necessitates that 
special attention be paid to the diagonal contribution restricted to minor arcs.

\begin{lemma}\label{lemma7.1} One has $u(0;\grm)=6B^3+O(B^{8/3})$.
\end{lemma}

\begin{proof} On recalling (\ref{2.11}), we find that
\[
u(0;[0,1))=u_{3,1}(0)=6B^3+O(B^{8/3}).
\]
In view of (\ref{7.1}), we see that $\text{mes}(\grM)=O(B^{2\del-3})$, meanwhile, and 
hence we deduce via orthogonality that
\begin{align*}
u(0;\grM)&=\int_\grM \sum_{\substack{1\le x_i,y_i\le B\\ 
x_1+x_2+x_3=y_1+y_2+y_3}}e(\alp(x_1^3+x_2^3+x_3^3-y_1^3-y_2^3-y_3^3))\d\alp \\
&\ll B^5\text{mes}(\grM)\ll B^{2+2\del}.
\end{align*}
Thus we conclude that
\[
u(0;\grm)=u(0;[0,1))-u(0;\grM)=6B^3+O(B^{8/3}).
\]
This completes the proof of the lemma.
\end{proof}

A similarly crude estimate for the major arc contribution handles $N(B;\grM)$.

\begin{lemma}\label{lemma7.2}
One has $N(B;\grM)\ll B^{4+2\del+\eps}$.
\end{lemma}

\begin{proof}
By orthogonality, it follows from (7.2) that
\[
N(B;\grM)=\sum_{\bfx,\bfy}\int_\grM e\Bigl(\alp \sum_{i=1}^5(x_i^3-y_i^3)\Bigr) \d\alp ,
\]
where the summation is over $5$-tuples $\bfx$, $\bfy$ with $1\le x_i,y_i\le B$ subject to 
the conditions
\[
\sum_{i=1}^3(x_i-y_i)=\sum_{i=4}^5(x_i^2-y_i^2)=0.
\]
The number of choices for $x_i$ and $y_i$ $(1\le i\le 3)$ is plainly $O(B^5)$. Meanwhile, 
by applying Hua's lemma (see \cite[Lemma 2.5]{Vau1997}) on a by now well-trodden path, 
the number of choices for $x_j$, $y_j$ $(j=4,5)$ is $O(B^{2+\eps})$. Thus we deduce via 
the triangle inequality that
\[
N(B;\grM)\ll B^5\cdot B^{2+\eps}\text{mes}(\grM)\ll B^{7+\eps}\cdot B^{2\del-3},
\]
and the conclusion of the lemma follows.
\end{proof}

We are now equipped to establish the final case of Theorem \ref{theorem1.1}.

\begin{lemma}\label{lemma7.3}
One has
\[
N_{3,2,1}(B)-12B^5\ll B^{5-\del/4+\eps}.
\]
\end{lemma}

\begin{proof} In view of (\ref{7.3}) and Lemma \ref{lemma7.2}, we have
\[
N_{3,2,1}(B)=N(B;\grm)+O(B^{4+2\del+\eps}).
\]
Then by (\ref{7.4}), we deduce that
\[
N_{3,2,1}(B)=u(0;\grm)v(0)+\Xi+O(B^{4+2\del+\eps}),
\]
where
\[
\Xi=\sum_{1\le |h|\le 2B^3}u(h;\grm)v(h).
\]
By wielding (\ref{2.9}) in combination with Lemma \ref{lemma7.1}, we may conclude thus 
far that
\[
N_{3,2,1}(B)=(6B^3+O(B^{8/3}))(2B^2+O(B))+O(B^{4+2\del+\eps})+\Xi,
\]
whence
\begin{equation}\label{7.5}
N_{3,2,1}(B)-12B^5\ll B^{14/3}+\Xi.
\end{equation}

\par Next we recall Lemma \ref{lemma5.1} and apply the inequalities of Cauchy and Bessel 
to obtain the upper bound
\begin{equation}\label{7.6}
\Xi^2\ll B^{3+\eps}\sum_{|h|\le 2B^3}|u(h;\grm)|^2\ll B^{3+\eps}
\int_\grm \Bigl( \int_0^1 |f_{3,1}(\alp,\bet)|^6\d\bet \Bigr)^2 \d\alp .
\end{equation}
As a consequence of Weyl's inequality (see \cite[Lemma 2.4]{Vau1997}), one has
\[
\sup_{\alp \in \grm}\sup_{\bet\in \dbR}|f_{3,1}(\alp,\bet)|\ll B^{1-\del/4+\eps}.
\]
Thus, by making use of orthogonality and \cite[Theorem 1.1]{BW2019}, we obtain the 
bound
\begin{align*}
\int_\grm \Bigl( \int_0^1 |f_{3,1}(\alp,\bet)|^6\d\bet \Bigr)^2 \d\alp 
&\ll (B^{1-\del/4+\eps})^2\int_{[0,1)^3}|f_{3,1}(\alp,\bet)^6f_{3,1}(\alp,\gam)^4|
\d\bfalp \\
&=B^{2-\del/2+2\eps}N_{3,1,1}(B)\\
&\ll B^{7-\del/2+2\eps}.
\end{align*}
By substituting this estimate into (\ref{7.6}), we arrive at the bound 
$\Xi\ll B^{5-\del/4+\eps}$, and hence (\ref{7.5}) delivers the relation
\[
N_{3,2,1}(B)-12B^5\ll B^{14/3}+B^{5-\del/4+\eps}.
\]
The conclusion of the lemma follows on recalling our hypothesis that $\del$ is any positive 
number smaller than $1/3$.
\end{proof}

This completes the proof of the last case of Theorem 1.1, namely case (v). We now discern 
via Lemmata \ref{lemma3.2}, \ref{lemma4.2}, \ref{lemma5.3}, \ref{lemma6.2} and 
\ref{lemma7.3} that in cases (i) to (v) we have the estimate (\ref{2.13}). Thus, as 
discussed in the sequel to that equation, the conclusion of Theorem \ref{theorem1.1} is 
confirmed.

\bibliographystyle{amsbracket}
\providecommand{\bysame}{\leavevmode\hbox to3em{\hrulefill}\thinspace}

\end{document}